\documentclass[12pt,a4paper,oneside,draft,hrule]{article}
\usepackage[hmargin=3cm,vmargin=3cm]{geometry}
\usepackage{amssymb}
\usepackage{amsmath}
\usepackage{amsthm}
\usepackage{amsfonts}
\usepackage{graphicx}
\usepackage{euscript}
\usepackage{amscd}
\usepackage{exscale}
\usepackage{fancyhdr}
\usepackage{setspace}
\numberwithin{equation}{section}
\newcommand{\nt}{\noindent}
\newcommand{\RN}{\mathbb{R}^N}

\newcommand{\RR}{\mathbb{R}}
\newcommand{\pu}{\partial U}
\newtheorem{tm}{Theorem}[section]
\newtheorem{prop}{Proposition}[section]
\newtheorem{lem}{Lemma}[section]
\def\blfootnote{\xdef\thefnmark{}\footnotetext}

\title{Positive harmonic functions on comb-like domains \blfootnote{This research was supported by Science Foundation Ireland under Grant/RFP/MAT057, and is also part of the programme of the ESF Network ``Harmonic and Complex Analysis and Applications'' (HCAA).}}
\author{Joanna Pres \footnote{School of Mathematical Sciences, University College Dublin, Belfield, Dublin 4, Ireland; joanna.t.pres@gmail.com; joanna.pres@ucdconnect.ie}}
\date{}

\begin{document}
\maketitle
\begin{abstract} This paper investigates positive harmonic functions on a domain which contains an infinite cylinder, and whose boundary is contained in the union of parallel hyperplanes. (In the plane its boundary consists of two sets of vertical semi-infinite lines.) It characterizes, in terms of the spacing between the hyperplanes, those domains for which there exist minimal harmonic functions with a certain exponential growth.



\end{abstract}

\section{Introduction}
The subtle relationship between the structure of  positive harmonic  functions on a domain $\Omega$ in $\RN$ ($N\geq 2$) and  boundary geometry has been much studied. One avenue of investigation has been to examine the effect of modifying the boundary of a familiar domain such as a half-space,  cone or cylinder. Thus many authors have been led to investigate the case of Denjoy domains $\Omega$, where the complement $\RN\setminus\Omega$ is contained in a hyperplane, say $\mathbb{R}^{N-1}\times\{0\}$ (see \cite{ben, che, g1989, anc2, seg1, seg2, cartot, cagar, and, lom}). For example,
Benedicks \cite{ben} has established a harmonic measure criterion that describes when the cone of positive harmonic functions on $\Omega$ that vanish on the boundary $\partial\Omega$ is generated by two linearly independent minimal harmonic functions.  (We recall that a positive harmonic function $h$ on a domain $\Omega$ is called \emph{minimal} if any non-negative harmonic minorant of $h$ on $\Omega$ is proportional to $h$.) Benedicks' criterion is also equivalent to the existence of a harmonic function $u$ on $\Omega$ vanishing on $\partial\Omega$ and satisfying $u(x)\geq |x_N|$ on $\Omega$, and thus describes when a Denjoy domain behaves like the union of two half-spaces from the point of view of potential theory.
 Related results, based on sectors, cones or cylinders, may be found in \cite{cransal, lom, gp}. The purpose of this paper is to describe what happens in the case of another relative of the infinite cylinder. More precisely, let $(a_n)$ be a strictly increasing sequence of non-negative numbers such that $a_n\rightarrow +\infty$ and $a_{n+1}-a_n\rightarrow 0$ as $n\rightarrow \infty$,
and let $B'$ be the unit ball in $\RR^{N-1}$.
We define $$E=\bigcup_{n\in \mathbb{N}}(\RR^{N-1}\setminus B')\times \{a_n\}$$ and investigate when the domain $\Omega=\RN \setminus E$ inherits the potential theoretic character of the cylinder $U=B'\times \RR$; that is, when the set $E$ imitates $\pu$ in terms of its effect on the asymptotic behaviour of positive harmonic functions on $\Omega$. We call such domains $\Omega$ \emph{comb-like} because they are a generalization of comb domains in the plane.

Let $x=(x^{\prime},x_{N})$ denote a typical point of Euclidean space
$\RN=\mathbb{R}^{N-1}\times \mathbb{R}$.
It is known (see \cite{g}, for example) that the cone of positive harmonic functions on $U$ that vanish on $\pu$ is generated by two minimal harmonic functions $h_{\pm }(x^{\prime },x_{N})=e^{\pm \alpha
x_{N}}\phi (x^{\prime })$, where $\alpha $ is the square root of
the first eigenvalue of the operator
$-\Delta=-\sum_{j=1}^{N-1}\partial^2/\partial x_j^2$  on $B'$ and
$\phi$ is the corresponding eigenfunction, normalized by
$\phi(0)=1$.
We describe below when  a comb-like domain  admits a minimal harmonic function $u$  that  vanishes on $\partial \Omega$ and satisfies $u\geq h_+$ on $U$.

\begin{tm}\label{big}Let $\nu>1$. Assume that $(a_n)$ satisfies the following condition
\begin{equation}\label{r} \frac{1}{\nu}\leq \frac{a_{k+1}-a_k}{a_{j+1}-a_j}\leq \nu
\end{equation}
whenever $|a_k-a_j|<4$.
The following statements are equivalent:
\begin{enumerate}
\item[(a)] there exists a positive harmonic function $u$ on $\Omega$ that satisfies $u\geq h_+$ on $U$ and $u$ vanishes continuously on $E$;
\item[(b)]
$\sum_{n=1}^\infty (a_{n+1}-a_n)^2<+\infty$.
\end{enumerate}
Moreover, if (b) holds, then $u$ can be chosen to be minimal in part (a).
\end{tm}
We will prove Theorem \ref{big} by combining methods from \cite{g1989}, \cite{cransal} and \cite{gp} with some new ideas. It is known (see \cite{burk, car, g1991}) that the behaviour of  minimal harmonic functions on  simply connected domains is intimately related to the classical  angular derivative problem. We note that when $N=2$,  condition (b) of Theorem \ref{big} is necessary and sufficient for the comb domain
$\Omega$ to have an angular derivative at $+\infty$ (see \cite{rodwar1, rodwar2, jen}).

\section{Notation and preliminary results}
We use $\partial^\infty D$ to denote the boundary of a domain $D$ in compactified space $\RN\cup\{\infty\}$.
Let $B_\rho(x)$ denote the open ball in $\RR^N$ of centre
$x$ and radius $\rho>0$.  We write $B'_\rho$ (resp.
$B_\rho$) for the open ball in $\RR^{N-1}$ (resp. $\RR^N$) of centre $0$ and radius $\rho$, and $V(\rho)=\partial
B'_\rho\times \RR$.  If $\rho=1$, we write $B'$ instead of
$B'_1$.
For $0<\rho_1<\rho_2$ let $A(\rho_1,
\rho_2)=\left(B'_{\rho_2}\setminus \overline{B'_{\rho_1}}\right)\times \RR$.
We denote by $\mu_x^D$ the harmonic
measure for an open set
$D\subset \RN$ evaluated at $x\in D$. If $f$ is a function
defined on $\partial^{\infty} D$, we use $\overline{H}_f^D$  to
denote the upper Perron-Wiener-Brelot solution to the Dirichlet
problem on $D$ and $H_f^D$  for the PWB solution of the
Dirichlet problem on $D$ when it exists. We  denote by $P_D(\cdot, y)$
the Poisson kernel for $D$ with pole $y\in\partial D$, where
$\partial D$ is smooth enough for it to be defined.
If $W\subseteq D$ and $u$ is a superharmonic function on $D$,
we denote by $R^W_u$ (resp. $\widehat R^W_u$)  the reduced
function (resp. the regularized reduced function) of $u$ relative
to $W$ in $D$. We denote surface area measure on a given surface by $\sigma$.
We use $C(a,b,...)$ to denote a constant depending at most on $a,b,...$,
the value of which may change from line to line.

 For the remainder of the paper, we  fix $0<r<1<R$ and for $x\in \pu$ we define $F_x=\partial B'_r\times [x_N-1,x_N+1]$ and $$T_x=(B'_R\setminus \overline{B'_r})\times (x_N-1, x_N+1 ).$$

We  note that the first eigenfunction $\phi$ of $-\Delta$ in $B'$ is comparable with the distance to  $\partial B'$, that is
\begin{equation}\label{phi}
C_1(N)(1-|x'|)\leq \phi(x')\leq C_2(N) (1-|x'|)\quad(x'\in B').
\end{equation}
A simple proof of \eqref{phi} can be found in \cite[pp.\! 419-420]{gaha}.
The following estimate for the Poisson kernel  (see \cite[Section 2.1]{gp}, for example) will prove useful. For $|x'|=s<1, x_N\in\RR, y\in\partial U$
\begin{equation}\label{pinterior}
 C_1(N,s) e^{-\alpha |x_N-y_N|}\leq P_U(x,y)\leq C_2(N,s)  e^{-\alpha |x_N-y_N|}.
\end{equation}
If $0<r_1<s<r_2$, similar estimates hold for $P_{A(r_1, r_2)}$ with $\alpha$ replaced by the square root of
the first eigenvalue of $-\Delta$ in $B'_{r_2}\backslash
\overline{B'_{r_1}}$ and constants $C_1, C_2$ depending on
$N, r_1, r_2$ and  $s$.

\begin{prop}\label{aimpliesb} Assume there exists a positive harmonic function $u$ on $\Omega$ such that $u\geq h_+$ on $U$ and $u$ vanishes on $E$. Then
\begin{equation}\label{series}\sum_{n=1}^{\infty}(a_{n+1}-a_n)^2<+\infty.
\end{equation}
\end{prop}

\begin{proof} By \eqref{pinterior} we have
\begin{equation}\label{pois}
\begin{aligned}
+\infty>u(0)\geq \int_{\pu}u(y) P_U(0,y)d\sigma(y)\geq C(N)\sum_{n=1}^{\infty}\int_{\partial B'\times (a_n,a_{n+1})} u(y)e^{-\alpha y_N}d\sigma(y).
\end{aligned}
\end{equation}
We use Harnack's inequalities and \eqref{phi} to see that for $y\in \pu$ with $$y_N\in\big(a_n+(a_{n+1}-a_n)/4, a_{n+1}-(a_{n+1}-a_n)/4\big)$$ the following holds
\begin{align}
u(y)&\geq C(N)u\left((1-(a_{n+1}-a_n)/8)y',y_N\right)\nonumber \\
&\geq C(N)e^{\alpha y_N}\phi\left((1-(a_{n+1}-a_n)/8)y'\right) \nonumber \\
& \geq C(N)e^{\alpha y_N}(a_{n+1}-a_n). \label{haru}
\end{align}
We deduce from \eqref{pois} and \eqref{haru} that \eqref{series} holds.
\end{proof}
\medskip

Assume now that $\sum_{n=1}^\infty (a_{n+1}-a_n)^2<+\infty$. Let $J\in \mathbb{N}$ be large enough so that $a_{n+1}-a_n\leq 1/2$ for $n\geq J$. For ease of exposition we rename the sequence $(a_n)_{n=J}^{\infty}$ as $(b_n)_{n=1}^{\infty}$. We also define $\rho_n=(b_{n+1}-b_n)/2$ for $n\in \mathbb{N} $. We introduce $b_0=b_1-1$ and $\rho_0=1/2$. For  technical reasons, we will work with
\begin{equation*}
E''=\bigcup_{n=1}^\infty\left(\RR^{N-1}\setminus B'\right)\times \{b_n\} \ \ \mbox{and}\ \  E'=\left(\partial B'\times (-\infty, b_1]\right)\cup E'',
\end{equation*}
and at the end we will dispense with these additional requirements.

\begin{lem} \label{beta1} There exists a positive constant $c_1$, depending on $N,R$ and $r$, such that for any $x\in\pu$ we have
\begin{equation}\label{beta}
\mu_x^{T_x\backslash E''}(F_x)\leq \mu_x^{T_x\backslash E''}(\partial
{T_x})\leq c_1\ \mu_x^{T_x\backslash E''}(F_x).
\end{equation}
\end{lem}
\begin{proof} Let $x\in \pu$. The left hand inequality in \eqref{beta} is obvious since  $F_x\subset \partial
T_x$.
Let $h=H^{T_x}_{\chi_{F_x}}$ on $T_x$ and $h=\chi_{F_x}$ on $\partial T_x$. In order to establish the right hand inequality, it is enough to prove that
\begin{equation}\label{2.2} h\leq h(x) \quad \textrm{on } E''\cap T_x.
\end{equation}
We will borrow an argument from \cite[Lemma 2.1]{gp}. Using  reflection in $\RR^{N-1}\times\{x_N+1\}$ to extend $h$ to $(\overline{B'_R}\setminus B'_r)\times[x_N-1, x_N+3]$, and translation, for $y\in \partial B'\times(x_N, x_N+1)$ we obtain
\begin{equation*}
\begin{aligned}
h(y)&=H^{T_x+(0',y_N-x_N)}_h(y)\\&=\mu_x^{T_x}\left(\partial B'_r\times [x_N-1, 2x_N+1-y_N]\right)-\mu_x^{T_x}\left(\partial B'_r\times [2x_N+1-y_N, x_N+1]\right)\\&\leq \mu_x^{T_x}(F_x)=h(x).
\end{aligned}
\end{equation*}
By symmetry, $h(y)\leq h(x)$ for $y\in \partial B'\times(x_N-1,x_N+1)$.
Since $$h(y)=0\leq h(x) \mbox{ for }y\in [\partial B'_R\times (x_N-1,x_N+1)]\cup[(B'_R\setminus \overline{B'})\times\{x_N-1,x_N+1\}],$$ using the maximum principle, we see that $h\leq h(x)$ on $(B'_R\setminus \overline{B'})\times(x_N-1,x_N+1)$, which proves \eqref{2.2}.
\end{proof}

We note that Lemma \ref{beta1} holds in a more general context when $E''$ is a closed subset of $\RN\setminus U$.

To prove Theorem \ref{big} we shall need the following estimate.
\begin{lem} \label{beta2}Let $\nu>1$. Assume that $(b_n)$ satisfies
\begin{equation}\label{ratio}\frac{1}{\nu}\leq \frac{b_{k+1}-b_k}{b_{j+1}-b_j}\leq \nu
\end{equation}
whenever $|b_k-b_j|<4$.
Then there exists a constant $c_2$, depending only on $N$, $r$ and $\nu$, such that
$$\mu_x^{T_x\setminus E''}(F_x)\leq c_2(b_{n+1}-b_n)$$
whenever $x\in \partial B'\times(b_n,b_{n+1})$ and $n\in \mathbb{N}$.
\end{lem}
\begin{proof} We suppose that $x\in \partial B'\times(b_n,b_{n+1})$ for some $n\in \mathbb{N}$. We define $\omega=(B'_R\setminus \overline{B'_r})\times (b_{j_0}, b_{k_0})$, where $j_0=\max\{j:b_j\leq x_N-1\}$ and $k_0=\min\{j: b_j\geq x_N+1\}$. Let $g=H^{\omega\setminus E''}_{\chi_{V(r)}}$ on $\omega\setminus E''$ and define $g=\chi_{V(r)}$ elsewhere. Let $m=\sup_{\pu}g$. We note that
$$\mu^{T_x\setminus E''}_x(F_x)\leq \mu^{\omega\setminus E''}_x({V(r)})\leq m.$$
We will obtain an upper bound for $m$ in terms of $\rho_n$.  To do this,  we define an open set $Z$ as follows
$$Z=\omega\setminus \bigcup_{k=0}^{\infty}\bigcup_{p\in[b_k,b_{k+1}]}\{z\in \overline{B'_s}\times \{p\}: s=(1-r)(|p-(b_k+\rho_k)|-\rho_k)+1\}.$$
We estimate $g$ on $\partial Z$ in terms of $m$ and $\rho_n$. Since $g=0$ on $\partial Z\setminus U$, we  estimate $g$ on $\partial Z\cap U$, noting that, for $y\in \omega\cap U$, we have
\begin{equation}\label{gfir} g(y)=H^{\omega\cap U}_g(y)=H^{\omega\cap U}_{\chi_{V(r)}}(y)+\int_{\pu\cap\omega} g d\mu^{\omega\cap U}_y.
\end{equation}
Let $g_1(y)=H^{\omega\cap U}_{\chi_{V(r)}}(y)$ and $g_2(y)=\int_{\pu\cap\omega} g d\mu^{\omega\cap U}_y$ for $y\in \omega\cap U$.  Using the function
\begin{equation*}f_N(y)=\left\{
\begin{array}{ll}
|y'|^{3-N}-1& (N\geq 4)\\
-\log|y'|& (N=3)\\
1-|y'|& (N=2)
\end{array}\right.
\end{equation*}
and the maximum principle, we find that for $y\in \partial Z\cap U$
\begin{equation}\label{noweg1} g_1(y)\leq f_N(y)/f_N(rx) \leq C_1(N,r)(1-|y'|)\leq  C_1(N,r,\nu)\rho_n.
\end{equation}
We now wish to show that there exists a constant $C_2(N,\nu)\in (0,1)$ such that
\begin{equation}\label{g22}g_2\leq C_2(N,\nu)m \ \ \mbox{on}\ \ \partial Z\cap U.
\end{equation}

Let $l=(1-r)\min_{j_0\leq k\leq k_0-1}\rho_k$ and let $t=(1,0,...,0,t_N)$ with $t_N\in \{b_k: k=j_0+1,...,k_0-1\}$.
By  \cite[Lemma 8.5.1]{armgar}, for $x\in B_{l/2}(1+l,0,...,0,t_N)$ we have $g(x)\leq C(N)(g(p_+)+g(p_-))$, where $p_\pm=(1+l,0,...,0,t_N\pm l/2)$. Using a Harnack chain to cover the longer arc joining $p_+$ and $p_-$ along the circle $\partial B_{\sqrt{5}l/2}(t)\cap (\RR\times\{0\}^{N-2}\times \RR)$, we deduce that $g\leq C_3(N)m$ on that circle. By the invariance of  $g$ under rotations around the $x_N$-axis and the maximum principle, this inequality holds on a torus-shaped set enclosing the edge of $(\RR^{N-1}\setminus B')\times\{t_N\}$; more precisely on every closed ball centred at a point of $\partial B'\times\{t_N\}$ and having radius $\sqrt{5}l/2$. When $t_N=b_{j_0}$ or $t_N=b_{k_0}$, this inequality follows directly from  \cite[Lemma 8.5.1]{armgar}, with a perhaps different constant, $C_4(N)$ say.
In particular, for  $y\in B^t\setminus E^t$, where $B^t=B_{\sqrt{5}l/2}(t)$, $E^t=[1,+\infty)\times\RR^{N-2}\times\{t_N\}$ and $t_N\in\{b_k: k=j_0,...,k_0\}$, we have

\begin{equation}\label{gsub}
g(y)\leq H^{B^t\setminus E^t}_g(y)=\int_{\partial B^t}gd\mu^{B^t\setminus E^t}_y\leq \max\{C_3(N), C_4(N)\}m H^{B^t\setminus E^t}_{\chi_{\partial B^t}}(y).
\end{equation}
Since $t$ is a regular boundary point for $B^t\setminus E^t$, there exists $\delta=\delta(N)>0$ such that
\begin{equation}\label{Hest}H^{B^t\setminus E^t}_{\chi_{\partial B^t}}(y)\leq \frac{1}{2\max\{C_3(N), C_4(N)\}}\ \ \ \ (y\in B_{\delta l}(t)\setminus E^t).
\end{equation}
Let $K_{\delta l}=\bigcup_{k=j_0}^{k_0}\{y\in\pu: |y_N-b_k|<\delta l\}$. In view of \eqref{gsub} and \eqref{Hest}, and the invariance of $g$ under rotations around the $x_N$-axis, we conclude that $g\leq m/2$ on $K_{\delta l}$.

Hence, for $y\in \partial Z\cap U$, we have
\begin{equation}\label{g2}
\begin{aligned}
g_2(y)\leq \int_{\pu} g d\mu^U_y\leq \frac{m}{2}\mu^U_y(K_{\delta l})+m \mu^U_y(\pu\setminus K_{\delta l})\leq m\left(1-\frac{1}{2}\mu^U_y(K_{\delta l})\right).
\end{aligned}
\end{equation}
We now show that there exists a  constant $C_5(N,\nu)\in (0,1)$ such that $\mu^U_y(K_{\delta l})\geq C_5(N,\nu)$ for $y\in \partial Z\cap U$. We first estimate $\mu^U_y(K_{\delta l})$ on some ball centred at $t$ and then join other points of $\partial Z\cap U$ by a Harnack chain.

Let $W_{\delta l}=B'\times (t_N-\delta l,t_N+\delta l)$.
 We use a dilation $\psi(y)=t+(y-t)/(\delta l)$
and note that, by continuity, there exists an absolute positive constant $\gamma$ such that for $y\in \psi(W_{\delta l})\cap B_\gamma(t)$ the following inequalities hold
\begin{equation*}H^{W_{\delta l}}_{\chi_{K_{\delta l}}}(\psi^{-1}(y))=H^{\psi(W_{\delta l})}_{\chi_{\psi(K_{\delta l})}}(y)\geq H^{(-\infty,1)\times\RR^{N-2}\times(t_N-1,t_N+1)}_{\chi_{\{1\}\times\RR^{N-2}\times[t_N-1,t_N+1]}}(y)\geq 1/2.\end{equation*}
Hence
\begin{equation*}\mu^U_y(K_{\delta l})\geq\mu^{W_{\delta l}}_y(K_{\delta l})\geq 1/2 \ \ \ (y\in B_{\gamma \delta l}(t)\cap U),\end{equation*}
and by Harnack's inequalities
$$\mu^U_y(K_{\delta l})\geq C_5(N,\nu)\ \ \ \mbox{for all} \ \ y\in \partial Z\cap U.$$
Let $C_2(N,\nu)=1-C_5(N,\nu)/2$. Then \eqref{g22} holds in view of \eqref{g2}, and by \eqref{gfir} and \eqref{noweg1} we have
$$g\leq C_1(N,r,\nu)\rho_n+C_2(N,\nu)m \ \ \mbox{on}\ \ \partial Z.$$
By the maximum principle this inequality holds on $Z$ and implies that
\begin{equation*}m\leq \frac{C_1(N,r,\nu)}{1-C_2(N,\nu)} \rho_n.\end{equation*}
This finishes the proof of lemma.
\end{proof}
\medskip

We define $\beta_{E'}(x)$ to be the harmonic measure of $\partial
T_x$  in $T_x\backslash E'$ evaluated at $x$. If $x\in E'$, then
$\beta_{E'}(x)$ is interpreted as $0$.
We observe that, if $(b_n)$ satisfies the ratio condition \eqref{ratio}, then,  in view of Lemmas \ref{beta1} and \ref{beta2},  we have
\begin{align}
\int_{\partial B'\times(b_1,+\infty)}\beta_{E'}(y)d\sigma (y)
&\leq c_1 \int_{\partial B'\times(b_1,+\infty)}\mu^{T_y\setminus E''}_y(F_y) d\sigma (y)\nonumber\\
&\leq c_1 c_2 \sigma_{N-1} \sum_{n=1}^{\infty}(b_{n+1}-b_n)^2,\label{gamlam}
\end{align}
where $\sigma_{N-1}$ denotes the surface measure of $\partial B'$ in $\RR^{N-1}$.

 Henceforth let $(b_n)$ satisfy \eqref{ratio}  and let $$\Lambda=\sum_{n=1}^{\infty}(b_{n+1}-b_n)^2<+\infty.$$
Before we prove the next lemma, we collect together some facts about certain Bessel functions (see \cite[Section 4]{armfug}). Let $K=K_{(N-3)/2}:(0,\infty)\rightarrow (0,\infty)$ denote the Bessel function of the third kind, of order $(N-3)/2$. Then  the function \begin{equation}\label{h0} h_0(x',x_N)=|x'|^{(3-N)/2}K(\pi |x'|)\sin(\pi x_N)
\end{equation}
 is positive and superharmonic on the strip $\RR^{N-1}\times(0,1)$, harmonic on $(\RR^{N-1}\setminus\{0'\})\times (0,1)$ and  vanishes on $\RR^{N-1}\times \{0,1\}\setminus \{(0',0), (0',1)\}$. Moreover, there exists $c(N)\geq 1$ such that
 \begin{equation}\label{Bessel} c(N)^{-1}\leq (2t/\pi)^{1/2}e^tK(t)\leq c(N)\ \ \mbox{for} \ t\in [1, +\infty).
 \end{equation}
We also recall a result of Domar (\cite[Theorem 2]{domar}).  Suppose that $D$ is a domain
in $\RN$ and $F:D\rightarrow [0,+\infty]$ is a given upper
semicontinuous function on $D$. Let
$\mathcal{F}$ be the collection of all subharmonic functions $u$, such that $u\leq
F$ on $D$. Domar's result says that if
\begin{equation}\label{domar}
\int_D [\log^+ F (x)]^{N-1+\varepsilon} dx < \infty,
\end{equation}
for some $\varepsilon>0$, then the function $M (x)= \sup_{u\in
\mathcal{F}}u(x)$ is bounded on every compact subset of $D$.

Let $0<r'<\min\{r,1/2\}$. Define $V=A(r',\infty)\setminus E'$ and $U_n=(\RR^{N-1}\setminus \overline{B'})\times (b_n,b_{n+1})$ for $n\in \mathbb{N}$.

\begin{lem} \label{ins}  There exists a positive constant $c_3$, depending on $N, R, r$ and $r'$, such that, for any positive harmonic function $u$ on $V$ that is bounded on each $U_n$ and vanishes on $E'$,
$$u(y)\leq c_3 u(rx',x_N)H^{T_x\setminus E'}_{\chi_{\partial T_x}}(y) \ \ \ (x\in \pu, y\in T_x\setminus E').$$
In particular, $$u(x',x_N)\leq c_3 \beta_{E'}(x) u(rx',x_N)\ \ \ (x\in \pu).$$
\end{lem}

\begin{proof} Let $x\in\pu$, $l=(1+r')/3$ and $L=2R$. Define $A_{x}=\{y: l<|y'|<L, |x_N-y_N|<2\}$. We  will show that
\begin{equation}\label{F}\frac{u(y)}{C(N,r,r')u(rx',x_N)}\leq F(y)\ \ \ \ (y\in A_x),
\end{equation}
where
\begin{equation*}F(y)=\left\{
\begin{array}{ll}
|1-|y'||^{1-N},& |y'|\neq1\\
+\infty,& |y'|=1.
\end{array}\right.
\end{equation*}

\nt \textit{Step 1.} Let $(y',y_N)\in A_x\cap U$. Harnack's inequalities yield that
\begin{equation*}
\begin{aligned}
u(y)\leq C(N,r,r')u(rx',x_N)(1-|y'|)^{1-N}.
\end{aligned}
\end{equation*}

\nt \textit{Step 2.} If $y\in A_x\cap U_n$ and $|y'|-1\leq \min\{y_N-b_n, b_{n+1}-y_N\}$, then there is a Harnack chain of  fixed length joining $(y',y_N)$ with $((2-|y'|)y'/|y'|, y_N)\in A_x\cap U$. By Step 1, we have
$$u(y)\leq C(N)u((2-|y'|)y'/|y'|, y_N)\leq C(N,r,r')u(rx',x_N)(|y'|-1)^{1-N}.$$

\nt \textit{Step 3.}  If $y\in A_x\cap U_n$ and $\rho_n\geq  |y'|-1> \min\{y_N-b_n, b_{n+1}-y_N\}$, we apply  \cite[Lemma 8.5.1]{armgar} and Harnack's inequalities to see that
$$u(y)\leq C(N)u(y',\widetilde{y}_N),$$
where $\widetilde{y}_N$ is such that $|\widetilde{y}_N-y_N|<|b_n+\rho_n-y_N|$ and $|y'|-1= \min\{\widetilde{y}_N-b_n, b_{n+1}-\widetilde{y}_N\}$. By Step 2,
 $$u(y)\leq C(N,r,r')u(rx',x_N)(|y'|-1)^{1-N}.$$

\nt \textit{Step 4.} If $y\in A_x\cap U_n$ and $|y'|\geq 1+\rho_n$, let $V_n=\{(z',z_N): 1+\rho_n<|z'|, z_N\in (b_n,b_{n+1})\}$.
For $z\in U_n$ we define a function
$$h_n(z)=\frac{h_0((z',z_N-b_n)/(2\rho_n))}{K\left(\pi(1+\rho_n)/(2\rho_n)\right)}\left(\frac{1+\rho_n}{2\rho_n}\right)^{(N-3)/2}
$$
which is harmonic on $U_n$ and vanishes on $\partial U_n \setminus \pu$.
Applying  \cite[Lemma 8.5.1]{armgar} and Harnack's inequalities to $u$ and $h_n$, by Step 3, we get
$$u(z)\leq C(N,r,r') u(rx',x_N) \rho_n^{1-N}h_n(z)\ \ \ \mbox{for} \ \ z\in \partial V_n.
$$
Since $u$ is bounded on $V_n$ and $\infty$ has zero harmonic measure  for $V_n$,
\begin{equation}u(y)\leq C(N,r,r')u(rx',x_N)\rho_n^{1-N} h_n(y).\label{hn}\end{equation}
Furthermore, by \eqref{h0} and \eqref{Bessel}
 \begin{equation*}
\begin{aligned}h_n(y)&\leq \left(\frac{1+\rho_n}{|y'|}\right)^{\frac{N-3}{2}}K\left(\frac{\pi |y'|}{2\rho_n}\right)\left(K\left(\frac{\pi(1+\rho_n)}{2\rho_n}\right)\right)^{-1}\\&
\leq C(N)e^{-\frac{\pi}{2\rho_n}(|y'|-1)}\left(\frac{1+\rho_n}{|y'|}\right)^{\frac{N-2}{2}}\\
&\leq C(N)e^{-\frac{\pi}{2\rho_n}(|y'|-1)}\\
&\leq C(N)\left(\frac{|y'|-1}{\rho_n}\right)^{1-N}.
\end{aligned}
\end{equation*}
Hence we see from \eqref{hn} that
$$u(y)\leq C(N,r,r')u(rx',x_N)(|y'|-1)^{1-N}.$$

We conclude that \eqref{F} follows from Steps 1-4.
 Since
\begin{equation*}
\int_{A_x}(\log^+F(y))^Ndy\leq C(N,R),
\end{equation*}
 Domar's result  and Harnack's inequalities  yield (if $r<l$)
\begin{equation*}
u(y)\leq C(N,R,r,r')u(rx',x_N)\ \ \ (y\in \overline{T_x}).
\end{equation*}
Therefore
$$u(y)=H^{T_x\setminus E'}_u(y)\leq C(N,R,r,r')u(rx',x_N)H^{T_x\setminus E'}_{\chi_{\partial T_x}}(y) \ \ \ (y\in T_x\setminus E').$$
In particular,
$$u(x)\leq C(N,R,r,r')u(rx',x_N)\beta_{E'}(x).$$
\end{proof}

\begin{lem}\label{vV} Let $v:\RN\cup\{\infty\}\rightarrow [0,+\infty]$ be a Borel measurable function such that $v(x)\leq e^{\alpha x_N}\chi_{V(r')}(x)$ on $\RN$.
 There exist positive constants $c_4$ and $c_5$, depending on $N,R,r$ and $r'$, such that, if $\Lambda \leq c_4$, then $H_v^V$ exists and
$$H^V_v(x)\leq H^{A(r',1)}_v(x)+c_5 \Lambda  e^{\alpha x_N}\ \ \ (|x'|=r).$$
\end{lem}
\begin{proof} Let $h_n=H^V_{\min\{v, n\}}$ on $V$ and
$h_n=\min\{v, n\}$ on $\partial^{\infty}V$, and let
\begin{equation*}
m_n=\sup\{e^{-\alpha x_N}h_n(x',x_N): |x'|=r, \ x_N>-n\}.
\end{equation*}
Then
\begin{equation}\label{change}
h_n=H^{A(r',1)}_{h_n}=H_{h_n\chi_{\pu}}^{A(r',1)}+H_{h_n\chi_{V(r')}}^{A(r',1)} \ \ \mbox{in}\ \ A(r',1).
\end{equation}  Let
$\alpha_{r'}>0$ denote the square root of the first eigenvalue of
$-\Delta$ in $B'\setminus\overline{B'_{r'}}$. Then $\alpha<\alpha_{r'}$ because the complement  of  $B'\setminus\overline{B'_{r'}}$ in $B'$ is non-polar (see \cite[Section 1.3.2]{hen}). Since $d\mu_x^{A(r',1)}=P_{A(r',1)}(x,\cdot)d\sigma$ on $\pu$, the Poisson kernel estimates yield, for $|x'|=r$,  that
\begin{equation*}
\begin{aligned}
e^{-\alpha x_N}H_{h_n\chi_{\pu}}^{A(r',1)}(x)&\leq C(N,r,r') e^{-\alpha x_N}\int_{\partial U}h_n(y)e^{-\alpha_{r'} |x_N-y_N|}d\sigma(y)\\
&\leq C(N,r,r') \int_{\partial U}h_n(y)e^{-\alpha y_N}d\sigma(y).
\end{aligned}
\end{equation*}
Noting that $h_n$ satisfies the hypotheses of Lemma \ref{ins}, we see from \eqref{gamlam} that, when $|x'|=r$ we have
\begin{equation}\label{gmn}
\begin{aligned}
e^{-\alpha x_N}H_{h_n\chi_{\pu}}^{A(r',1)}(x)\leq C(N,R,r,r') \int_{\partial U}e^{-\alpha y_N}h_n(ry',y_N)\beta_{E'}(y)d\sigma(y)\leq C_1 m_n \Lambda,
\end{aligned}
\end{equation}
where $C_1$ is a constant depending on $N,R,r,r'$ and $\nu$.

Moreover, for $|x'|=r$ we have
\begin{align}
e^{-\alpha x_N}H_{h_n\chi_{V(r')}}^{A(r',1)}(x)&\leq e^{-\alpha x_N}\int_{V(r')}e^{\alpha y_N}d\mu^{A(r',1)}_x(y)\nonumber\\
&\leq  C(N,r,r') \int_{V(r')}e^{\alpha (y_N-x_N)}e^{-\alpha_{r'} |y_N-x_N|}d\sigma(y)\nonumber\\
&\leq C(N,r,r')\int_{-\infty}^{+\infty}e^{(\alpha-\alpha_{r'}) |y_N-x_N|}dy_N\leq C_2(N,r,r').\label{term2}
\end{align}
By \eqref{change}-\eqref{term2} we obtain
\begin{equation*}
\begin{aligned}
e^{-\alpha x_N}h_n(x)  =e^{-\alpha x_N}H_{h_n\chi_{\pu}}^{A(r',1)}(x)+e^{-\alpha x_N}H_{h_n\chi_{V(r')}}^{A(r',1)}(x)\leq C_1 m_n \Lambda + C_2\ \ \ (|x'|=r).
\end{aligned}
\end{equation*}
Taking $c=\max\{C_1,C_2\}$ we arrive at
$$m_n\leq c(1+m_n\Lambda).$$ We choose $c_4=(2c)^{-1}$ and suppose that $\Lambda\leq c_4$. Then
$$m_n\leq c+m_n c c_4= c+m_n/2,$$
which implies that $m_n\leq 2c$.

It follows from \eqref{change} and \eqref{gmn}  that for $|x'|=r$ we have
\begin{equation}\label{final}e^{-\alpha x_N}h_n(x)\leq 2c^2\Lambda+e^{-\alpha x_N}H_{h_n\chi_{V(r')}}^{A(r',1)}(x).
\end{equation}
We choose $c_5=2c^2$ and let $n\rightarrow\infty$. By \eqref{term2} the limit of the latter term on the right hand side of \eqref{final} is finite and so $H^V_v$ exists and satisfies
$$H^V_v(x)\leq c_5 \Lambda  e^{\alpha x_N}+ H^{A(r',1)}_v(x)\ \ \ (|x'|=r).$$\end{proof}

\begin{lem}\label{Hw} Let $w:\partial^{\infty}U\rightarrow [0, +\infty)$ be a Borel measurable function such that
\begin{equation}\label{ww}
w(y)\leq \beta_{E'}(y)e^{\alpha y_N} \ \ \ (y\in\partial U)\ \  \mbox{ and } \ \ w(\infty)=0.
\end{equation} Then, there exists a positive constant $c_6$, depending on
$N,R,r$ and $\nu$, such that
$$H_w^U(x',x_N)\leq c_6
e^{\alpha x_N} \Lambda \ \ \ \ \ (|x'|=r).$$
\end{lem}
\begin{proof} Using \eqref{pinterior}, in view of \eqref{ww} and \eqref{gamlam}, for $|x'|=r$ we have
\begin{equation*}
\begin{aligned}
H_w^U(x',x_N)&\leq C(N,r) \int_{\partial U} w(y)e^{-\alpha|y_N-x_N|}d\sigma(y)\\
&\leq C(N,r) e^{\alpha x_N}\int_{\partial U}\beta_{E'}(y)
d\sigma(y)\\
&\leq C(N,R,r,\nu)   e^{\alpha x_N}\Lambda.
\end{aligned}
\end{equation*}
\end{proof}

We extend $h_+$ to be $0$ outside $U$ and recall that  $V$ stands for $A(r',\infty)\setminus E'$. We define inductively a sequence $(s_k)$ as follows
\begin{equation*}s_{-2}=s_{-1}= 0,\quad s_0=h_+,\end{equation*}
\begin{displaymath}s_{2k+1}=\left\{
\begin{array}{ll}
\overline H^V_{s_{2k}}& \  \textrm{on } V\\
s_{2k}&  \ \textrm{on } \RN\backslash V
\end{array}\right.,
\ \ \ \ \ \  s_{2k+2}=\left\{
\begin{array}{ll}
\overline H^U_{s_{2k+1}}+h_+& \  \textrm{on } U\\
s_{2k+1}& \ \textrm{on } \RN\setminus U
\end{array}\right. .
\end{displaymath}
We put $s_k(\infty)=0$ for all $k$.

\begin{lem}\label{finite}There is a positive constant $c_7$, depending on $N,R,r,r'$ and $\nu$, such that, if $\Lambda\leq c_7\lambda$ for some $\lambda\in(0,1)$, then:
\begin{enumerate}
\item[(a)] $(s_k)$ is an increasing sequence of continuous functions on $\RN$;
 \item[(b)] each $s_k$ is bounded on $\RR^{N-1}\times (-\infty, b_n)$ for each $n\in \mathbb{N}$;
\item[(c)] for all
$k=0,1,...$ we have
\begin{equation*}
\begin{aligned}
&(s_{2k}-s_{2k-2})(x)\leq \lambda^{k} e^{\alpha x_N},\ \ \ &
|x'|=r.
\end{aligned}
\end{equation*}
\end{enumerate}
\end{lem}

\begin{proof} We will use ideas from \cite[Lemma 3.1]{gp}. Suppose that $\Lambda\leq c_7\lambda$,
where $c_7$ is to be determined later.  Assume that
 $s_0\leq s_1\leq ...\leq s_{2k}$ on $\RN$ for some $k\geq 0$,  that all the functions $s_{k'}$ are continuous on $\RN$ for $0\leq k'\leq 2k$, and that for  $0\leq k'\leq k$
 \begin{equation}
\begin{aligned}
&(s_{2k'}-s_{2k'-2})(x',x_N)\leq \lambda^{k'} e^{\alpha x_N}\
&(|x'|=r).\label{2.14}
\end{aligned}
\end{equation}
We also fix $n\in \mathbb{N}$ and assume that $s_{2k}$ is bounded on $\RR^{N-1}\times (-\infty,b_n)$.  Once the terms of $(s_k)$ are seen to be finite, it is clear that the upper PWB solutions appearing in their definitions
are actually well defined PWB solutions. The induction  hypotheses clearly hold for $k=0$. We split the proof of
Lemma \ref{finite} into three steps.

\nt \textit{Step 1.} We show that $s_{2k+1}$ is a finite-valued continuous function on $\RN$ which is bounded on $\RR^{N-1}\times (-\infty, b_n)$.
Harnack's inequalities and \eqref{2.14} yield the existence of a constant
$c_8=c_8(N,r,r')>0$ such that
\begin{equation}
\begin{aligned}
&(s_{2k}-s_{2k-2})(y)\leq c_8 \lambda^{k} e^{\alpha y_N}\ \
\ (|y'|=r').\label{har}
\end{aligned}
\end{equation}
Now, for $|x'|=r$, by \eqref{har} and Lemma \ref{vV} we have
\begin{equation*}
\begin{aligned}
(s_{2k+1}-s_{2k-1})(x)&\leq \overline H_{s_{2k}-s_{2k-2}}^V(x)\\&
=\overline H_{(s_{2k}-s_{2k-2})\chi_{V(r')}}^V(x)\\&
\leq c_5c_8\lambda^k\Lambda e^{\alpha x_N}+ H_{(s_{2k}-s_{2k-2})\chi_{V(r')}}^{A(r',1)}(x).
\end{aligned}
\end{equation*}
Since $s_{2k}-s_{2k-1}=0$ on $\pu$ and $s_{2k}-s_{2k-1}=s_{2k}-s_{2k-2}$ on $V(r')$, it follows that $s_{2k}-s_{2k-1}$ belongs to the upper class for $H_{(s_{2k}-s_{2k-2})\chi_{V(r')}}^{A(r',1)}$. Hence
$$(s_{2k+1}-s_{2k-1})(x)\leq c_5c_8\lambda^k\Lambda e^{\alpha x_N}+(s_{2k}-s_{2k-1})(x),$$
and so
\begin{equation}\label{s2kplus1}(s_{2k+1}-s_{2k})(x)\leq c_5c_8\lambda^k\Lambda e^{\alpha x_N}\ \ \ (|x'|=r).
\end{equation}
This proves finiteness of $s_{2k+1}$.

A result of Armitage concerning a strong type of regularity for the PWB solution of the Dirichlet problem (see \cite[Theorem 2]{arm}) implies that $s_{2k+1}$ is continuous at points of $\partial V\setminus \bigcup_{n=1}^\infty(\partial B'\times\{b_n\})$.  Applying Lemma \ref{ins} to $v_j=H^V_{\min\{s_{2k},j\}}$ and $x\in \bigcup_{n=1}^\infty(\partial B'\times\{b_n\})$ we obtain
$$v_j(y)\leq c_3 v_j(rx',x_N)H^{T_x\setminus E'}_{\chi_{\partial T_x}}(y)\ \ \ (y\in T_x\setminus E').$$
Letting $j\rightarrow \infty$ we notice that the same inequality holds for $s_{2k+1}$, and hence the regularity of $x$
 for $T_x\setminus E'$ implies that $s_{2k+1}$ vanishes at $x$. We conclude that $s_{2k+1}$ is continuous on $\RN$.

 We also have $s_{2k+1}=H_{s_{2k+1}}^{V\cap [\RR^{N-1}\times (-\infty, b_n)]}$ on $V\cap [\RR^{N-1}\times (-\infty, b_n)]$. Further, since $s_{2k+1}$ is continuous on $\overline{B'}\times \{b_n\}$, vanishes on $E$ and is bounded on $(\RN\setminus V)\cap [\RR^{N-1}\times (-\infty, b_n)]$ in view of the induction hypothesis, we deduce that $s_{2k+1}$ is bounded above on $\RR^{N-1}\times (-\infty, b_n)$.

\nt \textit{Step 2.} We now prove that $s_{2k}\leq s_{2k+1}\leq s_{2k+2}$ on $\RN$. We note that  $s_{2k}=H^{A(r',1)}_{s_{2k}}$ on $A(r',1)$ (for a simple proof see Step 2 in the proof of \cite[Lemma 3.1]{gp}).

 It follows immediately from the induction hypothesis, that
$$s_{2k+1}=H^V_{s_{2k}}\geq H^V_{s_{2k-2}}=s_{2k-1}\ \ \ \mbox{on}\ \ V.$$
In particular, this gives
$s_{2k+1}\geq s_{2k}$ on $\RN \setminus U$. Hence, $s_{2k+1}\geq s_{2k}$
on $\partial U \cup
\partial V$. Using \cite[Theorem 6.3.6]{armgar}, we  obtain
$$s_{2k+1}=H^V_{s_{2k}}= H^{A(r',1)}_{s_{2k+1}}\geq H^{A(r',1)}_{s_{2k}}=s_{2k} \ \ \textrm{ on }\ \ A(r',1).$$
Therefore, $s_{2k+1}\geq s_{2k}$ on
$\RN$. We now  deduce that
$$s_{2k+2}=\overline H^U_{s_{2k+1}}+h_+\geq H^U_{s_{2k-1}}+h_+=s_{2k}=s_{2k+1}\ \textrm{ on } \ \RN\setminus V.$$
 We finally note that, if $s_{2k+2}$ belongs to the upper class for $\overline{H}^{A(r',1)}_{s_{2k+2}}$, we obtain
$$s_{2k+2}\geq \overline{H}^{A(r',1)}_{s_{2k+2}}\geq H^{A(r',1)}_{s_{2k+1}}=s_{2k+1}\ \ \ \mbox{on}\ \ A(r',1),$$ and so $s_{2k+2}\geq s_{2k+1}$ on $\RN$.
To verify that $s_{2k+2}$ belongs to the upper class for $\overline{H}^{A(r',1)}_{s_{2k+2}}$ it is enough to check that $\liminf_{x\rightarrow y}s_{2k+2}(x)\geq s_{2k+2}(y)$ for $y\in \pu$. This is clear from regularity and the continuity of $s_{2k+1}$, as if $s_{2k+2}\not\equiv+\infty$, then for $y\in\pu$ we have
\begin{equation*}
\begin{aligned}
\liminf_{x\rightarrow y}s_{2k+2}(x)=\liminf_{x\rightarrow y}H^U_{s_{2k+1}}(x)\geq
\liminf_{x\rightarrow y, x\in\pu}{s_{2k+1}}(x)=s_{2k+1}(y)=s_{2k+2}(y).
\end{aligned}
\end{equation*}

\nt \textit{Step 3.} In the final step we will prove that
\begin{equation}\label{diffe}
(s_{2k+2}-s_{2k})(x)\leq \lambda^{k+1} e^{\alpha x_N} \ \ \ (|x'|=r).
\end{equation}
Then, using \cite[Theorem 2]{arm}, we can conclude that $s_{2k+2}$ is continuous on $\RN$. Further,  $s_{2k+2}-h_+=H^U_{s_{2k+1}}=H^{U\cap [\RR^{N-1}\times (-\infty, b_n)]}_{s_{2k+2}-h_+}$ on $U\cap [\RR^{N-1}\times (-\infty, b_n)]$. By continuity, $s_{2k+2}$ is bounded on $\overline{B'}\times\{b_n\}$. On $\RN\setminus U$ we have $s_{2k+2}=s_{2k+1}$,  which is bounded on $(\RN\setminus U)\cap[\RR^{N-1}\times (-\infty, b_n)]$ by Step 1. Hence $s_{2k+2}$ is bounded on the whole of $\RR^{N-1}\times (-\infty, b_n)$.

 To prove the desired inequality \eqref{diffe}, we first recall that  $$U_m=(\RR^{N-1}\setminus \overline{B'})\times (b_m,b_{m+1}) \ \ \ (m\in \mathbb{N}).$$ Noting that $$s_{2k+1}=H^V_{s_{2k}}=H^{U_m}_{s_{2k+1}}=H^{U_m}_{s_{2k+1}\chi_{\partial B'\times(b_m,b_{m+1})}}\ \ \ \mbox{on} \ \ U_m,$$
 and that, by continuity, $s_{2k+1}$ is bounded on $\partial B'\times(b_m,b_{m+1})$, we see that $s_{2k+1}-s_{2k-1}$ satisfies the hypotheses of Lemma \ref{ins}.
 Hence, for $x\in\pu$, we have
 $$\begin{aligned}
(s_{2k+1}-s_{2k-1})(x)& \leq c_3 \beta_{E'}(x)(s_{2k+1}-s_{2k-1})(rx',x_N)\\&
=c_3 \beta_{E'}(x)[(s_{2k+1}-s_{2k})(rx',x_N)+(s_{2k}-s_{2k-1})(rx',x_N)]\\
&\leq c_3 \beta_{E'}(x)[(s_{2k+1}-s_{2k})(rx',x_N)+(s_{2k}-s_{2k-2})(rx',x_N)].
\end{aligned}
$$
It follows from \eqref{s2kplus1} and our induction hypothesis that
$$(s_{2k+1}-s_{2k-1})(x)\leq c_3(c_5 c_8\Lambda+1)\lambda^{k}e^{\alpha x_N}\beta_{E'}(x) \ \ (x\in\pu).$$
Assuming that $c_7\leq 1$ and letting $c_9=c_3(c_5 c_8+1)$ we obtain
$$(s_{2k+1}-s_{2k-1})(x)\leq c_9\lambda^{k}e^{\alpha x_N}\beta_{E'}(x) \ \ (x\in\pu).$$
 By Lemma \ref{Hw}, for $|x'|=r$, we have
 $$(s_{2k+2}-s_{2k})(x)\leq \overline{H}^U_{s_{2k+1}-s_{2k-1}}(x)\leq c_9\lambda^k c_6\Lambda e^{\alpha x_N}=c_6 c_7 c_9 \lambda^{k+1}e^{\alpha x_N}.$$
 Taking $c_7=\min\{1, (c_6c_9)^{-1}\}$ we find that \eqref{diffe} holds, and the proof is complete. \end{proof}

\section{Proof of Theorem \ref{big}}

 Proposition \ref{aimpliesb} gives the implication $(a)\Rightarrow (b)$.  To prove that $(b)\Rightarrow (a)$  we first observe  that taking $J$ large enough when setting $b_1=a_J$, we can ensure that $\Lambda\leq c_7\lambda$ for some $\lambda\in (0,1)$.
Let $\Omega'=\RN\setminus E'$ and $\displaystyle{u'=\lim_{k\rightarrow
\infty}s_{k}}$.  By Lemma \ref{finite}, for $|x'|=r$ we obtain
\begin{equation*}
s_{2k}(x)=\sum_{j=0}^{k}(s_{2j}-s_{2j-2})(x)\leq
\sum_{j=0}^{k}\lambda^j e^{\alpha x_N}\leq \frac{1}{1-\lambda}
e^{\alpha x_N}.
 \end{equation*}
Hence $u'\not\equiv+\infty$. As a limit of an increasing sequence
$(s_{2k})$ of harmonic functions on $U$, the function $u'$ is harmonic on $U$.  Since $u'$ is the limit of an increasing sequence
$(s_{2k+1})$ of harmonic functions on $V$, it is also harmonic on
$V$.
Hence $u'$ is harmonic in $\Omega'$. It  follows from the
monotonicity of $(s_k)$ that $u'\geq h_+$ on $U$.

For $x\in E'$ we have $u'(x)=0$. By the monotone convergence theorem applied to the equation $s_{2k+1}=H^V_{s_{2k}}$  we obtain $u'=H^V_{u'}$ on $V$. We can follow the reasoning from the second last paragraph of Step 1 in the proof of Lemma \ref{finite} to see that $u'$ vanishes continuously on $E'$.

We next prove that $u'$ is minimal on $\Omega'$ using an argument from \cite[Theorem 1.1]{gp}.
As a consequence of the monotone
convergence theorem we find that
\begin{equation}
u'(x)=H^U_{u'}(x)+h_+(x) \ \ \ \ (x\in U).\label{2.23}
\end{equation}
Let $\Delta_1$ denote the minimal Martin boundary of $\Omega'$ and let
$M$ be the Martin kernel of $\Omega'$  relative to the origin.
 By the Martin representation theorem (see \cite[Theorem 8.4.1]{armgar})
we have
\begin{equation}
u'(x)=\int_{\Delta_1}M(x,z)d\nu_{u'}(z)\ \ \ (x\in
\Omega'),\label{2.24}
\end{equation}
where $\nu_{u'}$ is uniquely determined by $u'$.

We define $T=\{z\in \Delta_1: \Omega'\backslash U\textrm{ is minimally
thin at }z\}$ so that
\begin{equation}
R_{M(\cdot, z)}^{\Omega'\backslash U}=M(\cdot,z)\ \ \ \
(z\in\Delta_1\setminus T).\label{2.25}
\end{equation}
Changing the order of integration, and using \eqref{2.23}-\eqref{2.25} and \cite[Theorem 6.9.1]{armgar},
 we obtain
\begin{equation*}
\begin{aligned}
h_+(x)&=\int_{\Delta_1}\left(
M(x,z)-\int_{\partial U} M(y, z)d\mu_x^U(y)\right) d\nu_{u'}(z)\\
&=\int_{\Delta_1}\left(
M(x,z)-R_{M(\cdot, z)}^{\Omega'\backslash U}(x)\right) d\nu_{u'}(z)\\
& =\int_{T}\left(M(x,z)-R_{M(\cdot, z)}^{\Omega'\backslash
U}(x)\right)d\nu_{u'}(z) \ \ \ \ (x\in U).
\end{aligned}
\end{equation*}
We now claim that  $\nu_{u'}|_T$ is concentrated at a single point. For the sake of contradiction suppose that there are two distinct points
$y_1,y_2\in \Delta_1\cap{\rm supp}(\nu_{u'}|_T)$  and let $N_1, N_2$ be disjoint neigbourhoods of $y_1$ and
$y_2$ respectively. We define
\begin{equation*}
h_j(x)=\int_{N_j\cap T}\left(M(x,y)-R_{M(\cdot,
y)}^{\Omega'\setminus U}(x)\right)d\nu_{u'}(y) \ \ \ \ (x\in \Omega',
j=1,2),
\end{equation*}
and note that $h_j\leq h_+$ on $U$. Minimality of $h_+$ on $U$ implies that
\begin{equation}\label{hjj}h_j/h_j(0)=h_+ \ \ \mbox{on}\  U \ \ (j=1,2).
\end{equation}
We now define
$$
v_j(x)=\int_{N_j\cap T}M(x,y)d\nu_{u'}(y) \quad (x\in\Omega', j=1,2).
$$
Then $h_j\leq v_j\leq u'$ on  $\Omega'$, and by \eqref{hjj},
$v_j/h_j(0)\geq h_+$ on $\Omega'$ ($j=1,2$). In view of  the
definition of $s_k$ we have $v_j/h_j(0)\geq s_k$ on $\Omega'$
for all $k\in \mathbb{N}$  and so $v_j/h_j(0)\geq u'$  on $\Omega'$
($j=1,2$). It follows that  $h_1(0)v_2\leq v_1$ on $\Omega'$.
This implies that $\nu_{u'}|_{T\cap N_1}$ is minorized by a multiple
of $\nu_{u'}|_{T\cap N_2}$, which contradicts the fact that $N_1\cap N_2=\emptyset$.
Hence  $\nu_{u'}|_{T}=c\delta_{t'}$ for some $t'\in T$ and $c>0$.
Furthermore, the minimal harmonic
function $v=cM(\cdot,t')$ on $\Omega'$ satisfies $u'\geq v$ on $\Omega'$ and $v\geq
h_+$ on $U$. We observe that $v\geq
s_k$ on $\Omega'$ for all $k\in \mathbb{N}$, and so $v\geq u'$. Hence
$v\equiv u'$ and we conclude that $u'$ is minimal on $\Omega'$.

Let $\Omega''=\RN \setminus E''$. We define $g=H^{\Omega'}_{\chi_{\Omega''\setminus \Omega'}}$ and $g=\chi_{\Omega''\setminus \Omega'}$ on $\partial^{\infty}\Omega'$. By \cite[Theorem 6.9.1]{armgar} we have
$g=R^{\Omega''\setminus\Omega'}_1$ on $\Omega''$ (reductions with respect to non-negative superharmonic functions on $\Omega''$). Since $\Omega''\setminus \Omega'$ is non-thin at each constituent point, it follows from \cite[Theorem 7.3.1(i)]{armgar} that $R^{\Omega''\setminus\Omega'}_1=\widehat{R}^{\Omega''\setminus\Omega'}_1$ on $\Omega''$ and so $g$ is superharmonic there. Let $h$ be a non-negative harmonic minorant of $g$ on $\Omega''$. Then $h$ is bounded on $\Omega''$ and vanishes quasi-everywhere on $\partial \Omega''$.  Since a polar subset of $\partial \Omega''$ and $\{\infty\}$ are both negligible for $\Omega''$ (see \cite[Theorems 6.5.5 and 7.6.5]{armgar}), we deduce that $h\equiv 0$. Hence $g$ is a potential on $\Omega''$.

Let $W=[\RR^{N-1}\times(-\infty,b_n)]\cap \Omega'$ for some $n>1$. Since $1-g$ is positive and continuous on $\overline{B'}\times\{b_n\}$, it follows that $1-g$ is bounded below by a positive constant on this set while $u'$ is bounded from above there. Hence there exists a positive constant $c$ such that $c(1-g)\geq u'$ on $\overline{B'}\times\{b_n\}$, and thus on $\partial W$.  By Lemma \ref{finite}(b) each $s_k$ is bounded on $W$ and so it belongs to the lower class for $H^W_{s_k}$. These facts combined with monotonicity of $(s_k)$ lead to the observation that
$$s_k\leq H^W_{s_k}\leq H^W_{u'}\leq c H^W_{1-g}=c(1-g) \ \ \mbox{on} \ \ W.$$
Therefore, $u'\leq c(1-g)$ on $W$. Since $c(1-g)-u'$ is a non-negative harmonic function on $W$ which vanishes on $\Omega''\setminus \Omega'$, we conclude that $c(1-g)-u'$ is subharmonic on $\Omega''$, so that $u'+cg$ is superharmonic on $\Omega''$.

By the Riesz decomposition,
\begin{equation}\label{rieszz}
u'+cg=u''+G_{\Omega''}\mu \ \ \mbox{ on} \ \ \Omega'',
\end{equation}
where $u''$ is the greatest harmonic minorant of $u'+cg$ on $\Omega''$ and $G_{\Omega''}\mu$ is the Green potential of the Riesz measure $\mu$ associated with  $u'+cg$. Hence $u''$ vanishes on $E''\setminus (\partial B'\times \{b_1\})$ and for each $n\in \mathbb{N}$ it is bounded on $\RR^{N-1}\times(-\infty, b_n)$. It follows from a removable singularity result (see \cite[Theorem 5.2.1]{armgar}) that $u''$ extends to a subharmonic function on $\RN$. This together with the non-thinness of $E''$ at points of $\partial B'\times \{b_1\}$ implies that $u''$ vanishes also on $\partial B'\times \{b_1\}$.
Since $h_+$ is a subharmonic minorant of $u'+cg$ on $\Omega''$, we deduce that $h_+\leq u''$ on $\Omega''$.

It remains to show that $u''$ is minimal. Let $h$ be a positive harmonic minorant of $u''$ on $\Omega''$. We notice that $h$ is bounded on $\Omega''\setminus \Omega'$ and vanishes on $\partial \Omega''$. Hence the greatest harmonic minorant of $R^{\Omega''\setminus \Omega'}_h$ on $\Omega''$ is bounded and vanishes on $\partial \Omega''$, and we see that $R^{\Omega''\setminus \Omega'}_h$ is a potential on $\Omega''$. Since the upper-bounded  harmonic function  $h-R^{\Omega''\setminus \Omega'}_h-u'$ on $\Omega'$ satisfies $$\limsup_{x\rightarrow y} (h-R^{\Omega''\setminus \Omega'}_h-u')(x)\leq 0 \ \ \ \mbox{for} \ \ y\in\partial \Omega',$$
and $\{\infty\}$ has zero harmonic measure for $\Omega'$, it follows that
$$h-R^{\Omega''\setminus \Omega'}_h-u'\leq 0 \ \ \mbox{on}\ \ \Omega'.$$

Now, since $h-R^{\Omega''\setminus \Omega'}_h$ is a positive harmonic minorant of the minimal function $u'$ on $\Omega'$, we conclude that $h-R^{\Omega''\setminus \Omega'}_h=au'$ for some $a\in(0,1]$. Substituting this into \eqref{rieszz} we obtain
$$h+acg=au''+aG_{\Omega''}\mu+R^{\Omega''\setminus \Omega'}_h \ \ \mbox{on } \Omega''.$$
Taking the greatest harmonic minorant in $\Omega''$ of both sides we get $h=au''$, which means that $u''$ is minimal.

Let $u=u''-H^\Omega_{u''}$.
 Since $u''-h_+\geq 0$ is superharmonic on $\Omega''$
 and equals $u''$ on $\Omega''\setminus\Omega$, we have
 $$u=u''-R^{\Omega''\setminus \Omega}_{u''}=u''-R^{\Omega''\setminus \Omega}_{u''-h_+}\geq h_+.$$
 Since the points of $\partial \Omega$ are regular for $\Omega$ and $u''$ is continuous, it follows that $u$ vanishes on $\partial \Omega$. Further, \cite[Theorem 9.5.5]{armgar} shows that $u$ is minimal.
\medskip

\nt \emph{Remark.} The proof of the implication $(a)\Rightarrow (b)$ in Theorem \ref{big} does not rely on condition \eqref{r}. It is in the proof of the converse that our methods rely on such a condition.  However, it is enough to assume merely that $\Omega$  is contained in a comb-like domain $\Omega_0$ for which \eqref{r} holds. To see this, suppose that $(b)$ holds. Theorem \ref{big} applied to $\Omega_0$ yields the existence of a minimal harmonic function $u_0$ on $\Omega_0$ which vanishes on $\partial \Omega_0$ and satisfies $u_0\geq h_+$. Let $u=u_0-H_{u_0}^{\Omega}$ on $\Omega$. The argument from the previous paragraph shows that $u$ is as stated in $(a)$.

\begin{singlespace}

\end{singlespace}

\begin{thebibliography}{00}
\bibitem{aik} Aikawa, H.: Positive harmonic functions of finite order in a Denjoy type
domain. Proc. Amer. Math. Soc. 131, 3873-3881 (2003).
\bibitem{anc2} Ancona, A.: Sur la fronti\`{e}re de Martin des domaines de
Denjoy. Ann. Acad. Sci. Fenn. Ser. A I Math. 15, 259-271 (1990).
\bibitem{and} Andrievskii, V.V.: Positive harmonic functions on Denjoy domains in the complex
plane. J. Analyse Math. 104, 83-124 (2008).
\bibitem{arm} Armitage, D.H.:
A strong type of regularity for the PWB solution of the Dirichlet
problem. Proc. Amer. Math. Soc.  61, 285-289 (1976).

\bibitem{armfug} Armitage, D.H., Fugard, T.B.:
Subharmonic functions in strips. J. Math. Anal. Appl.  89, 1-27  (1982).

\bibitem{armgar} Armitage, D.H., Gardiner, S.J.:
Classical potential theory. Springer, London (2001).

\bibitem{ben} Benedicks, M.: Positive harmonic functions vanishing on the boundary of certain domains in
$\RR^N$. Ark. Mat. 18, 53-72 (1980).


\bibitem{burk} Burdzy, K.: Brownian excursions and minimal thinness. Part III: Application to the angular derivative problem. Math. Z. 192, 89-107 (1986).


\bibitem{cartot} Carleson, L., Totik, V.:  H\"older continuity of Green's functions. Acta Sci. Math. (Szeged) 70, 557-608 (2004).

\bibitem{car} Carroll, T.:  A classical proof of Burdzy's theorem on the angular derivative. J. London Math. Soc. 38, 423-441 (1988).

\bibitem{cagar} Carroll, T.,  Gardiner, S.J.: Lipschitz continuity of the Green function in Denjoy domains. Ark. Mat. 46, 271-283
(2008).
\bibitem{che}  Chevallier, N.: Fronti\`{e}re de Martin d'un domaine de $\RR^N$ dont le bord est inclus dans une hypersurface lipschitzienne.
Ark. Mat. 27, 29-48 (1989).

\bibitem{cransal} Cranston, M.C.,  Salisbury, T.S.: Martin boundaries of sectorial domains. Ark. Mat.  31,
27-49 (1993).

\bibitem{domar} Domar, Y.: On the existence of a largest subharmonic minorant of a given function. Ark. Mat. 39,  429-440
(1957).

\bibitem{g1991} Gardiner, S.J.: A short proof of Burdzy's theorem on the angular derivative.  Bull. London Math. Soc.  23, No 6, 575-579 (1991).

\bibitem{g1989} Gardiner, S.J.: Minimal harmonic functions on Denjoy
domains. Proc. Amer. Math. Soc.  107,  963-970 (1989).


\bibitem{g} Gardiner, S.J.: The Martin boundary of NTA strips.
Bull. London Math. Soc.  22,  163-166 (1990).

\bibitem{gaha} Gardiner, S.J., Hansen, W.: The Riesz decomposition of finely superharmonic functions. Adv. Math.
214, 417-436 (2007).


\bibitem{gp} Ghergu, M., Pres, J.: Positive harmonic functions that vanish
on a subset of a cylindrical surface. Potential Anal. 31, 147-181 (2009).


\bibitem{hen} Henrot, A.: Extremum problems for eigenvalues of elliptic operators. Birkhäuser Verlag, Basel (2006).
\bibitem{jen} Jenkins, J.A.:
On comb domains.  Proc. Amer. Math. Soc.  124, 187-191 (1996).

\bibitem{lom} L{\"{o}}mker, A.: Martin boundaries of quasi-sectorial domains. Potential Anal. 13, 11-67
(2000).

 \bibitem{rodwar2} Rodin, B., Warschawski, S.E.: Angular derivative conditions for comb domains. Contemp. Math. 38, 61-68 (1985).
 \bibitem{rodwar1} Rodin, B., Warschawski, S.E.: Extremal length and univalent functions. The angular derivative. Math. Z. 153, 1-17 (1977).

\bibitem{seg1} Segawa, S.: Martin boundaries of Denjoy domains.
Proc. Amer. Math. Soc.  103,  177-183 (1988).

\bibitem{seg2} Segawa, S.: Martin boundaries of Denjoy domains  and quasiconformal
mappings. J. Math. Kyoto Univ.  30,  297-316 (1990).

\end{thebibliography}
\end{document}